\documentclass[final]{amsart}

\usepackage{amsmath}
\usepackage{amssymb}

\usepackage{color}

\newtheorem{theorem}{Theorem}[section]
\newtheorem{lemma}[theorem]{Lemma}
\newtheorem{proposition}[theorem]{Proposition}
\newtheorem{remark}[theorem]{Remark}

\numberwithin{equation}{section}

\DeclareMathOperator*{\diag}{diag}
\DeclareMathOperator*{\Arg}{Arg}
\newcommand{\e}{\mathrm{e}}
\newcommand{\dd}{\mathrm{d}}
\newcommand{\ii}{\mathrm{i}}

\newcommand{\Rset}{\mathbb{R}}
\newcommand{\Zset}{\mathbb{Z}}
\begin{document}

\title{On spectral assignment for neutral type systems}
\author{K.~V.~Sklyar}
\address{Institute of Mathematics, University of Szczecin\\
Wielkopolska 15, 70-451 Szczecin, Poland}
\curraddr{}
\email{sklar@univ.szczecin.pl}
\thanks{This work was 
supported in part by the Polish Nat. Sci. Center, grant
N N514 238 438}
\author{R.~Rabah}
\address{IRCCyN/\'Ecole des Mines de Nantes\\
4 rue Alfred Kastler, BP 20722, 44307 Nantes, France}
\curraddr{}
\email{rabah.rabah@mines-nantes.fr}
\thanks{}
\author{G.~M.~Sklyar}
\address{Institute of Mathematics, University of Szczecin\\
Wielkopolska 15, 70-451 Szczecin, Poland}
\curraddr{}
\email{sklar@univ.szczecin.pl}
\thanks{}

\date{\today}

\dedicatory{}

\begin{abstract}
For a large class of linear neutral type systems the problem of eigenvalues and eigenvectors assignment  
is investigated, i.e. finding the system which has the given spectrum and   almost all, in some sense,
 eigenvectors.
\end{abstract}

\maketitle
\section{Introduction}
One of central problems in control theory is the spectral assignment problem. This question is well
 investigated for linear finite dimensional systems. It is important to emphasize that 
 the assignment of eigenvalues is not sufficient in several cases. One needs also the 
 assignment of eigenvectors or of the geometric eigenstructure. For
 infinite dimensional problems (delay systems, partial derivative equations) the problem is much more complicated.
 
 Our purpose is to investigate this kind of problems for a large class of neutral type systems given by the equation
 \begin{equation}
\label{rownanie1.1}
\dot{z}(t)=A_{-1}\dot{z}(t-1)
+\int^{0}_{-1}A_2(\theta)\dot{z}(t+\theta)\dd\theta
+\int^{0}_{-1}A_3(\theta)z(t+\theta)\dd\theta,
\end{equation}
where $z(t) \in \Rset^n$ and $A_{-1}, A_2, A_3$ are $n\times n$ matrices. The elements of $A_2$ and $A_3$ taking 
values in $L_2(-1,0)$. The neutral type term $A_{-1}\dot{z}(t-1)$ consists on a simple delay, while the other include
as multiple as distributed delays. The behavior of such systems can be described mainly by the algebraic and 
geometric properties of the spectrum of the matrix $A_{-1}$ 
(cf. \cite{Rabah_Sklyar_Rezounenko_2003,Rabah_Sklyar_Rezounenko_2005}).

It is well known that then spectral properties of this system are described by the characteristic matrix $\Delta(\lambda)$ 
given by 
$$
\Delta(\lambda)=
\lambda
I-\lambda \e^{-\lambda}A_{-1}-\int^{0}_{-1}\lambda \e^{{\lambda}\theta}A_2(\theta)\dd\theta-
\int^{0}_{-1}\e^{{\lambda}\theta}A_3(\theta)\dd\theta.
$$
The eigenvalues are roots of the equation $\det\Delta(\lambda)=0$. The eigenvectors of the system (more precisely of the functional operator model of the system) are expressed through the matrix $\Delta(\lambda_k)$, where 
$\lambda_k$ is an eigenvalue. In fact the problem of an assignment of an infinite number 
of eigenvalues and eigenvectors
is reduced to a problem of assignment of singular values and degenerating vectors of 
an entire matrix value function $\Delta(\lambda)$. It is remarkable \cite{Rabah_Sklyar_Rezounenko_2005} 
that the roots of
$\det\Delta(\lambda)=0$ are quadratically close to a fixed set of complex number which 
are the logarithm of eigenvalues of the matrix $A_{-1}$. Moreover, the degenerating vectors of 
$\Delta(\lambda_k)$ are also quadratically close to the eigenvectors of the matrix $A_{-1}$.

In this paper we investigate an inverse problem:\\[1.5ex]
\textit{What conditions must satisfy a sequence of complex numbers $\{\lambda\}$ and a sequence 
of vectors $\{v\}$ in order to be a sequence of roots of the characteristic 
equation $\det\Delta (\lambda)=0$ and a sequence of
degenerating vectors of the characteristic matrix $\Delta(\lambda)$ of 
equation \eqref{rownanie1.1} respectively for some choice of matrices $A_{-1},A_{2}(\theta), A_{3}(\theta)$ ?}\\[1.5ex]
One of the possible application of this problem is to investigate a vector moment problem
via the solution of the exact controllability property for a corresponding neutral type system 
by a relation devlopped in \cite{Rabah_Sklyar_2007}.\\[2ex]
The present paper is a detailed version of the short note published in Comptes Rendus Mathematiques 
\cite{Sklyar_Rabah_Sklyar_2013}. 
\section{Operator form of perturbation}
We consider neutral type systems of the form 
\begin{equation}
\label{eq:1}
\dot{z}(t)=A_{-1}\dot{z}(t-1)
+\int^{0}_{-1}A_2(\theta)\dot{z}(t+\theta)\dd\theta
+\int^{0}_{-1}A_3(\theta)z(t+\theta)\dd\theta,
\end{equation}
where $A_{-1}$ is a constant $n\times n$ matrix, $A_{2},A_{3}$ are
$n\times n$ matrices
whose elements belong to $L_{2}[-1,0].$

As it is shown in \cite{Rabah_Sklyar_Rezounenko_2003}, \cite{Rabah_Sklyar_Rezounenko_2005} this system can be rewritten in the operator form
$$
\frac{\dd}{\dd t}\binom{y}{z_t(\cdot)}=\mathcal{A}\binom{y}{z_t(\cdot)},
$$
where $\mathcal{A} : D(\mathcal{A})\to M_{2}=\mathbb{C}^n\times L_{2}([-1,0],\mathbb{C}^n)$,
$$
D(\mathcal{A})=\left\{\binom{y}{\varphi(\cdot)}\mid \varphi(\cdot)\in H^1([-1,0],\mathbb{C}^n),\
y=\varphi(0)-A_{-1}\varphi(-1)\right\}\subset M_2,
$$ 
and the operator $\mathcal{A}$ is given by formula
$$
\mathcal{A}\binom{y}{\varphi(\cdot)}=\binom{\int^{0}_{-1}A_2(\theta)\dot{\varphi}(\theta)\dd\theta
+\int^{0}_{-1}A_3(\theta)\varphi(\theta)\dd\theta}{\frac{\dd\varphi}{\dd\theta}(\cdot)}.
$$
This operator is noted $\widetilde{\mathcal{A}}$ instead of $\mathcal{A}$ if $A_2(\theta)=A_3(\theta)\equiv0$.
The operator $\widetilde{\mathcal{A}}$ is defined on the same domain $D({\mathcal{A}})$. 
One can consider that
the state operator ${\mathcal{A}}$ is a perturbation of the operator $\widetilde{\mathcal{A}}$, namely
$$
\mathcal{A}\binom{y}{\varphi(\cdot)}=\widetilde{\mathcal{A}}\binom{y}{\varphi(\cdot)}+ 
\binom{\int^{0}_{-1}A_2(\theta)\dot{\varphi}(\theta)\dd\theta
+\int^{0}_{-1}A_3(\theta)\varphi(\theta)\dd\theta}{0}.
$$
Let $\mathcal{B}_0:\mathbb{C}^n\to M_2$ be given by
$$\mathcal{B}_0y=\binom{y}{0},$$
and $\mathcal{P}_0:D(\mathcal{A})\to\mathbb{C}^n$ by
\begin{equation}
\label{rownanie1.2}
\mathcal{P}_0\binom{y}{\varphi(\cdot)}=\int^{0}_{-1}A_2(\theta)\dot{\varphi}(\theta)\dd\theta
+\int^{0}_{-1}A_3(\theta)\varphi(\theta)\dd\theta.
\end{equation}
Then $\mathcal{A}=\widetilde{\mathcal{A}}+\mathcal{B}_0\mathcal{P}_0$. 
Denote by $X_{\mathcal{A}}$  the set $D(\mathcal{A})$ endowed with the graph norm.
Let us show that $\mathcal{P}_0$ browses the set of all linear bounded operators
$\mathcal{L}(X_{\mathcal{A}},\mathbb{C}^n)$ as $A_2(\cdot),A_3(\cdot)$ run over the set of
$n\times n$ matrices with components from $L_2[-1,0].$ 
Indeed, an arbitrary linear operator $Q$ from
$\mathcal{L}(X_{\mathcal{A}},\mathbb{C}^n)$ can be presented as
\begin{eqnarray*}
Q\binom{y}{\varphi(\cdot)}& = &Q_1y +Q_2{\varphi(\cdot)}\\
&=&Q_1(\varphi(0)-A_{-1}\varphi(-1))+ \int^{0}_{-1}\widehat{A}_2(\theta)\dot{\varphi}(\theta)\dd\theta
+\int^{0}_{-1}\widehat{A}_3(\theta)\varphi(\theta)\dd\theta,
\end{eqnarray*}
where $\widehat{A}_2(\cdot),\widehat{A}_3(\cdot)$ are $(n\times n)$-matrices with 
component from $L_2[-1,0]$ and $Q_1$ is a $(n\times n)$ matrix.
Let us observe that
\begin{eqnarray*}
\varphi(-1)&=&\int^{0}_{-1}\theta\dot{\varphi}(\theta)\dd\theta+\int^{0}_{-1}\varphi(\theta)\dd\theta,\cr
\varphi(0)&=&\int^{0}_{-1}(\theta+1)\dot{\varphi}(\theta)\dd\theta+\int^{0}_{-1}\varphi(\theta)\dd\theta 
\end{eqnarray*}
and denote
\begin{eqnarray*}
A_2(\theta)&=&\widehat{A}_2(\theta)+(\theta+1)Q_1-\theta Q_1A_{-1},\cr
A_3(\theta)&=&\widehat{A}_3(\theta)+Q_1-Q_1A_{-1}.
\end{eqnarray*}
Then, with these notations, the operator $Q$ may be written as
$$Q\binom{y}{\varphi(\cdot)}=\int^{0}_{-1}A_2(\theta)\dot{\varphi}(\theta)\dd\theta+
\int^{0}_{-1}A_3(\theta){\varphi}(\theta)\dd\theta.$$
Hence formula \eqref{rownanie1.2} describes all the operators from $\mathcal{L}(X_{\mathcal{A}},\mathbb{C}^n).$

\section{An equation for eigenvalues and eigenvectors of the characteristic matrix (spectral equation)}

Consider the operator $\mathcal{A}=\widetilde{\mathcal{A}}+\mathcal{B}_0\mathcal{P}_0$
and assume that $\lambda_0$ is an eigenvalue of $\mathcal{A}$ and $x_0$ is a corresponding eigenvector, i.e.
\begin{equation}
\label{rownanie2.1}
(\widetilde{\mathcal{A}}+\mathcal{B}_0\mathcal{P}_0)x_0=\lambda_0 x_0.
\end{equation}
Let us assume further that $\lambda_0$ does not belong to spectrum of
 $\widetilde{\mathcal{A}}$ and denote by $R(\widetilde{\mathcal{A}},\lambda_{0})=(\widetilde{\mathcal{A}}-\lambda_{0}I)^{-1},$
with this notation \eqref{rownanie2.1} reads as
\begin{equation}
\label{rownanie 2.2}
x_0+R(\widetilde{\mathcal{A}},\lambda_0)\mathcal{B}_0\mathcal{P}_0 x_0=0.
\end{equation}
Let us notice that $v_0=\mathcal{P}_0 x_0\neq 0,$ because $\lambda_0\notin\sigma(\widetilde{\mathcal{A}}).$ Then applying operator $\mathcal{P}_0$ to the left hand side of \eqref{rownanie 2.2} we get $$v_0+\mathcal{P}_0R(\widetilde{\mathcal{A}},\lambda_0)\mathcal{B}_0 v_0=0$$
This equality means that $\lambda_0$ is a point of singularity of the matrix-valued function
$F(\lambda)=I+\mathcal{P}_0R(\widetilde{\mathcal{A}},\lambda)\mathcal{B}_0$ and $v_0$ is a vector 
degenerating $F(\lambda_0)$ from the right.

Let $w_0^\ast$ be a non-zero row such that
\begin{equation}
\label{rownanie2.3}
w_0^\ast F(\lambda_0)=0.
\end{equation}
In order to describe the vector $w_0$, let us find first another form for the  matrix $F(\lambda)$.
For any $v\in\mathbb{C}^n$ we denote
$$
R(\widetilde{\mathcal{A}},\lambda)\mathcal{B}_0v=\binom{y}{\varphi(\cdot)},
\quad\lambda\notin\sigma(\widetilde{\mathcal{A}}),
$$
Then
$$(\widetilde{\mathcal{A}}-\lambda I)\binom{y}{\varphi(\cdot)}=\mathcal{B}_0v=\binom{v}{0}.$$
This gives 
$$
v=-\lambda y,\qquad \frac{\dd\varphi(\theta)}{\dd\theta}-\lambda\varphi(\theta)=0,
$$
and, as  $\binom{y}{\varphi(\cdot)}\in D(\widetilde{\mathcal{A}})$, 
we obtain $y=\varphi(0)-A_{-1}\varphi(-1).$
Therefore
$$
\varphi(\theta)=D\e^{\lambda\theta},\quad D\in\mathbb{C}^n,
$$
and
$$
v=-\lambda(I-\e^{-\lambda}A_{-1})D.
$$
Since the matrix $(I-\e^{-\lambda}A_{-1})$ is invertible $(\lambda\notin\sigma(\widetilde{\mathcal{A}})),$
then we get
$$\varphi(\theta)=-\frac{\e^{\lambda\theta}}{\lambda}(I-\e^{-\lambda}A_{-1})^{-1}v,
$$
and hence 
$$
R(\widetilde{\mathcal{A}},\lambda)\mathcal{B}_0v=
\begin{pmatrix}{y}\cr
{-\dfrac{\e^{\lambda\theta}}{\lambda}(I-\e^{-\lambda}A_{-1})^{-1}v}
\end{pmatrix}.
$$
This formula and \eqref{rownanie1.2} implies
\begin{eqnarray*}
\lefteqn{v+\mathcal{P}_{0}R(\widetilde{\mathcal{A}},\lambda)\mathcal{B}_{0}v=}\\
&& =v-\left(\int^{0}_{-1}\e^{{\lambda}\theta}A_2(\theta)\dd\theta-
\int^{0}_{-1}\frac{\e^{{\lambda}\theta}}{\lambda}A_3(\theta)\dd\theta\right)(I-\e^{-\lambda}A_{-1})^{-1}v
\end{eqnarray*}
and hence
$$F(\lambda)v=\Delta(\lambda)\frac{(I-\e^{-\lambda}A_{-1})^{-1}}{\lambda}v,$$
where
$$
\Delta(\lambda)=\lambda
I-\lambda \e^{-\lambda}A_{-1}-\int^{0}_{-1}\lambda \e^{{\lambda}\theta}A_2(\theta)\dd\theta-
\int^{0}_{-1}\e^{{\lambda}\theta}A_3(\theta)\dd\theta
$$
is the characteristic matrix of equation \eqref{rownanie1.1}.
Thus, the  equality \eqref{rownanie2.3} for $\lambda_{0}\notin\sigma(\widetilde{\mathcal{A}})$ 
is equivalent to 
$$
{w_{0}}^{*}\Delta(\lambda_{0})=0.
$$
Summarizing we obtain the following

\begin{proposition}\label{statement1} Let $\lambda_{0}$ do not belong to $\sigma(\widetilde{\mathcal{A}}).$ 
Then the pair
$(\lambda_{0},w_{0}),\:w_{0}\in\mathbb{C}^n,\:w_{0}\neq {0},$ satisfies  
equation \eqref{rownanie2.3} 
if and only if $\lambda_{0}$ is 
a root of the characteristic equation 
$$
\det\Delta(\lambda)=0
$$
and ${w_{0}}^{*}$ is a row-vector degenerating $\Delta(\lambda_{0})$ from the left, 
i.e. ${w_{0}}^{*}\Delta(\lambda_{0})=0.$
\end{proposition}
Thus, one can consider the equation $w^*F(\lambda)=0$ as an equation whose roots $(\lambda,w)=(\lambda_{0},w_0)$
describe all eigenvalues and (right) eigenvectors of the characteristic matrix $\Delta(\lambda).$

\section{A component-wise representation  of spectral equation}

We recall spectral properties of operators $\widetilde{\mathcal{A}}$ and
 $\widetilde{\mathcal{A}}^\ast$ obtained in \cite{Rabah_Sklyar_Rezounenko_2003,Rabah_Sklyar_Rezounenko_2005}. 
We will assume that the matrix $A_{-1}$ has 
a simple non-zero eigenvalues $\mu_{1},\mu_{2},\mu_{3},\ldots,\mu_{n}.$ In this case the spectrum 
$\sigma(\widetilde{\mathcal{A}})$ consists of simple eigenvalues which we denote by
$$
\widetilde\lambda^{m}_{k}=\ln\left|\mu_{m}\right|+\ii(\Arg \mu_{m}+2\pi k),\quad m=1,\ldots,n,\ \!k\in\mathbb{Z},
$$
and of eigenvalue $\widetilde{\lambda}_{0}=0.$ First assume $\mu_{m}\neq 1,\ m=1,\ldots,n.$
Then the corresponding to ${\widetilde{\lambda}^{m}_{k}}$ eigenvectors are of the form

\begin{equation}
\label{rownanie3.1}
\widetilde\varphi^{m}_{k}=\binom{0}{{\e^{{\widetilde{\lambda}^{m}_{k}}\theta}}y_{m}},
\quad k\in\mathbb{Z},\:m=1,\ldots,n,
\end{equation}
where $y_{1},\ldots,y_{n}$ are eigenvectors of $A_{-1}$ corresponding to $\mu_{1},\mu_{2},\ldots,\mu_{n}$.
The eigen\-space corresponding
to $\tilde{\lambda}_{0}=0$ is $n$ dimensional and its basis is

\begin{equation}
\label{rownanie3.2}
\widetilde\varphi^{0}_{j}=\binom{(1-\mu_{j})y_{j}}{y_{j}}, \quad j=1,\ldots,n.
\end{equation}
If some $\mu_{m}$ say $\mu_{1}$ equals 1, then $\tilde{\lambda}^{1}_{0}=\tilde{\lambda}_{0}=0.$ In that case 
the eigenspace corresponding to 0 is $(n+1)$-dimensional and its basis consists of $n$ 
eigenvectors $\widetilde\varphi^{0}_{j},\:j=1,\ldots,n,$ given by \eqref{rownanie3.2}, 
and one 
rootvector
$$
\widetilde\varphi^{1}_{0}=\binom{y_{1}}{\theta y_{1}},
\qquad \widetilde{\mathcal{A}}\widetilde\varphi^{1}_{0}
=\widetilde\varphi^{0}_{1}.
$$
All the vectors $\widetilde\varphi^{m}_{k}$ given by \eqref{rownanie3.1}, except of
$\widetilde\varphi^{0}_{1},$ are still eigenvectors of $\widetilde{\mathcal{A}}$ and 
correspond 
to eigenvalues  ${\widetilde \lambda^{m}_{k}}.$ In  both cases 
the 
family $\Phi=\left\{{\widetilde\varphi^{m}_{k}}\right\}\cup\left\{{\widetilde\varphi^{0}_{j}}\right\}$
forms a Riesz basis in the space $M_2.$
Denote by $\Psi=\left\{{\widetilde\psi^{m}_{k}}\right\}\cup\left\{{\widetilde\psi^{0}_{j}}\right\}$ 
the  bi-orthogonal basis to $\Phi$. Then
\begin{equation}
\label{rownanie3.3}
{\widetilde\psi^{m}_{k}}=
\begin{pmatrix}
{z_{m}}/{\overline{\widetilde\lambda_k^m}} \\[1.5ex]
{\e^{-\overline{\widetilde{\lambda}_{k}^{m}}\theta}z_{m}}
\end{pmatrix},
\qquad m=1,\ldots,n, \ k\in\mathbb{Z},\ \widetilde{\lambda}^{1}_{0}\neq 0,
\end{equation}
where $\overline{\widetilde{\lambda}}$ is the complex conjugate of $\widetilde{\lambda}$ 
and

$${\widetilde\psi^{1}_{0}}=\binom{z_{1}}{0}$$ if $\widetilde{\lambda}^{1}_0=0,$
here $z_{m}$ are eigenvectors of matrix $A_{-1}^{*}$ such that 
$$\left\langle y_{i},z_{j}\right\rangle=\delta_{i j}.$$
It is easy to see that $\widetilde\psi^{m}_{k}$ are eigenvectors of the operator $\mathcal{A}^{*}$ corresponding 
to eigenvalues $\overline{\widetilde\lambda_k^m}$.

 Our nearest goal is to rewrite the matrix $F(\lambda_{0})$ in the basis $\Phi$. 
 Let us first rewrite the expression
 $R(\widetilde{\mathcal{A}},\lambda_{0})\mathcal{B}_0$.
  We present $\mathcal{B}_0$ 
as a matrix $\mathcal{B}_0=(b_1^0,b_2^0,\ldots,b_n^0)$ with infinite columns $b_i^0$
 which are vectors from $M_2$ of the form
 $$b_i^0=\binom{e_i}{0}, \qquad i=1,\dots, n,$$ where
 $e_i$ is the canonical basis of $\mathbb{C}^n$. Then
$$
b_i^0=\sum_{k,m}\left\langle b_i^0,\widetilde{\psi}_{k}^{m}\right\rangle\widetilde{\varphi}_{k}^{m}
+\sum_j\left\langle b_i^0,\widetilde{\psi}_{j}^{0}\right\rangle\widetilde{\varphi}_{j}^{0}, \qquad i=1,\dots, n,
$$
and in the case $\mu_m\neq1,\ m=1,\ldots,n$, we have
$$
R(\widetilde{\mathcal{A}},\lambda_0)b_i^0=\sum_{k,m}\frac{1}{\widetilde\lambda_k^m
-\lambda_0}\left\langle b_i^0,\widetilde\psi_{k}^{m}\right\rangle\widetilde\varphi_{k}^{m}
-\sum_j\frac{1}{\lambda}_0\left\langle b_i^0,\widetilde\psi_{j}^{0}\right\rangle\widetilde\varphi_{j}^{0}, 
\quad i=1,\dots, n.
$$
Taking into account the form of eigenvectors \eqref{rownanie3.3} we find that for all $m=1,\dots,n$
$$
\left(\left\langle b_1^0,\widetilde\psi_{k}^{m}\right\rangle,\ldots,\left\langle b_n^0,
\widetilde\psi_{k}^{m}\right\rangle\right)=\frac{1}{\widetilde{\lambda}_k^m}(\left\langle e_1,z_m\right\rangle,\ldots,
\left\langle e_n,z_m\right\rangle)=z_m^{*}/\widetilde{\lambda}_k^m.
$$
So we obtain
\begin{equation}
\label{rownanie3.4}
R(\widetilde{\mathcal{A}},\lambda_0)\mathcal{B}_0
=\sum_{k,m}\frac{1}{\widetilde{\lambda}_k^m
-\lambda_0}\times \frac{1}{\widetilde{\lambda}_k^m}\widetilde\varphi^{m}_{k}z_m^{*}
-\sum_{j}\frac{1}{\lambda_0}\widetilde\varphi_{j}^{0}
\left(\left\langle b_1^0,\widetilde\psi_{j}^{0}\right\rangle,\ldots,\left\langle b_n^0,\widetilde\psi_{j}^{0}
\right\rangle\right),
\end{equation}
if all the numbers $\widetilde\lambda_k^m\neq 0.$ If $\mu_1=1$ and $\widetilde\lambda_0^1=0$ 
then in the sum \eqref{rownanie3.4} the term corresponding to $k=1,\:m=1,$ is replaced by
$$
\left(\frac{1}{\widetilde\lambda_0^1-\lambda_0}\widetilde\varphi_{0}^{1}
-\frac{1}{({\widetilde\lambda_0^1-\lambda_0})^{2}}\widetilde\varphi_{1}^{0}\right)z_1^*.
$$
In what follows, we shall use the notation
\begin{displaymath}
\widetilde\beta_k^m= \left\{ \begin{array}{ll}
\widetilde\lambda_k^m, & \widetilde\lambda_k^m\neq 0,\\
1, & \widetilde\lambda_k^m=0,
\end{array} \right.
\qquad m=1,\dots,n, \ k \in \Zset.
\end{displaymath}
Let us observe that the expression $w_0^*\mathcal{P}_0$ is a linear bounded functional 
on the space $X_{\mathcal{A}}$, i.e. $w_0^*\mathcal{P}_0\in\mathcal{L}(X_{\mathcal{A}},\mathbb{C})$.
The representation of $w_0^*$ in the basis $z_j$ is as follows:
$$w_0^*=\sum_{j}\alpha_jz_j^*.$$
Consider now $n$ functionals $z_j^*\mathcal{P}_0\in\mathcal{L}(X_{\mathcal{A}},\mathbb{C}),\:j=1,\ldots,n$. 
One can decompose them in the basis $\Psi$:
$$
z_j^*\mathcal{P}_0=\sum_{k,m}{p}_{k,m}^j\widetilde{\psi}_k^m+\sum_{i}
{\widetilde{p}}_i^j\widetilde{\psi}_i^0,$$
where
\begin{equation}
\label{rownanie3.5}
\sum_{k}\left|\frac{p_{k,m}^j}{\widetilde\beta_k^m}\right|^{2}<\infty, \qquad m=1,\dots,n.
\end{equation}
In the sequel, we shall assume $\widetilde{p}_i^j=0,\ i=1,\ldots,n$.
This means that from now we consider perturbations $\mathcal{P}_0$ satisfying the condition
\begin{equation}
\label{rownanie3.6}
\int^{0}_{-1}A_3(\theta)\dd\theta=0.
\end{equation}
Then we have
$$
w_0^*\mathcal{P}_0=\sum_{j}{{\alpha}_j} z_j^*\mathcal{P}_0=\sum_{j}
{\alpha}_j\sum_{k,m}{p}_{k,m}^j\widetilde{\psi}_k^m.
$$
From this relation and expression \eqref{rownanie3.4} we obtain the equality
\begin{eqnarray*}
w_0^*\mathcal{P}_0\mathcal{R}(\widetilde{\mathcal{A}},\lambda_0)\mathcal{B}_0&
=&\sum_m\left\langle\sum_k
\frac{1}{\widetilde{\beta}_k^m}\times \frac{1}{\widetilde\lambda_k^m-\lambda_0}\widetilde{\varphi}_k^m,
\sum_{j}\overline{{\alpha}_j}\sum_{k_0,m_0}\overline{\beta}_{k_0,m_0}^j\widetilde{\psi}_{k_0}^{m_0}\right\rangle z_m^*\\
&=&\sum_{j}\alpha_j\sum_{k_0,m_0}
\frac{p_{k,m}^j}{\widetilde{\beta}_k^m}\times\frac{1}{\widetilde\lambda_k^m- \lambda_0}z_m^*.
\end{eqnarray*}
With these notations, the equation \eqref{rownanie2.3} reads  
$$
0=w_0^*\left(I+\mathcal{P}_0\mathcal{R}(\widetilde{\mathcal{A}},
\lambda_0)\mathcal{B}_0\right)=\sum_{m}\alpha_m z_m^*+\sum_{m}
\left(\sum_{j,k}\alpha_j\frac{p_{k,m}^j}{\widetilde{\beta}_k^m}\times
\frac{1}{\widetilde\lambda_k^m-\lambda_0}\right)
z_m^*.$$ 
Thus, the condition for a pair $(\lambda_0,w_0)$ to satisfy the spectral equation can be rewritten 
in the form of the following system of $n$ equations:
\begin{equation}
\label{rownanie3.7}
\alpha_m=-\sum_{\stackrel{k\in\mathbb{Z}}{j=1,\ldots,n}}\alpha_j\left(\frac{p_{k,m}^j}{\widetilde{\beta}_k^m}
\times \frac{1}{\widetilde\lambda_k^m-\lambda_0}\right), \quad m=1,\ldots,n,
\end{equation}
where for any fixed couple $m,j$, the \textbf{needed }$n$-tuple $\left\{p_{k,m}^j\right\}$ satisfies \eqref{rownanie3.5}.
\section{Conditions for spectral assignment}
Now we discuss the following question:

\medskip \noindent
\textit{What conditions must satisfy a sequence of complex numbers $\{\lambda\}$ and a sequence 
of vectors $\{v\}$ in order to be a sequence of roots of the characteristic 
equation $\det\Delta (\lambda)=0$ and a sequence of
 degenerating vectors of the characteristic matrix $\Delta(\lambda)$ of 
equation \eqref{rownanie1.1} respectively for some choice of matrices $A_{-1},A_{2}(\theta), A_{3}(\theta)$?}

\medskip 
We will assume that the corresponding operator $\mathcal{A}$ has simple eigenvalues only. 
Let us remember 
that we assumed earlier that all eigenvalues of matrix $A_{-1}$ are also simple. 
Then one can enumerate 
those eigenvalues as
 $\left\{\lambda_k^m\right\}\cup\left\{\lambda_j^0\right\},\:m,j=1,\ldots,n;\:k\in\mathbb{Z},$
where (see \cite[Theorem 1]{Rabah_Sklyar_Rezounenko_2005}) the sequence $\left\{\lambda_k^m\right\}$ 
satisfies
\begin{equation}
\label{rownanie4.1}
\sum_{k,m}\left|\lambda_k^m-\widetilde{\lambda}_k^m\right|^{2}<\infty.
\end{equation}
Denote by $\left\{\varphi_k^m\right\}\cup\left\{\varphi_j^0\right\},\:m,j=1,\ldots,n;\:k\in\mathbb{Z},$ 
corresponding eigenvectors of $\mathcal{A}.$ 
Then (see \cite[Lemma 13, Theorem 15]{Rabah_Sklyar_Rezounenko_2005}) these vectors form a Riesz basis 
in $M_2$ which is quadratically 
close to the basis
$\left\{\widetilde{\varphi}_k^m\right\}\cup\left\{\widetilde{\varphi}_j^0\right\}$ if 
we assume that the corresponding elements have the same norm. 
Eigenvectors $\left\{\varphi\right\}$ 
has the form
\begin{displaymath}
\mathbf{\varphi}=
\left( \begin{array}{ccc}
v-\e^{-\lambda}A_{-1}v\\
\e^{\lambda\theta}v\\
\end{array} \right)
\end{displaymath}
with $\Delta(\lambda)v=0.$ Therefore the fact that the basis $\left\{\varphi\right\}$ and $\left\{\widetilde{\varphi}\right\}$ 
are quadratically close implies the condition
\begin{equation}
\label{rownanie4.2}
\sum_k{\left\|v_k^m-y_m\right\|}^2<\infty,\quad m=1,\ldots,n,
\end{equation}
where $\Delta(\lambda_k^m)v_k^m=0,\:m=1,\ldots,n;\:k\in\mathbb{Z}$, and $y_m$ 
are eigenvectors of $A_{-1}$ corresponding to $\mu_{m}$ as in (\ref{rownanie3.1}).
If we apply the same arguments for  the dual system
\begin{equation}
\label{rownanie4.3}
\dot{z}(t)=A_{-1}^*\dot{z}(t-1)+\int^{0}_{-1}A_2^*(\theta)\dot{z}(t+\theta)\dd\theta+
\int^{0}_{-1}A_3^*(\theta)z(t+\theta)\dd\theta,
\end{equation}
then we obtain the symmetric condition
\begin{equation}
\label{rownanie4.4}
\sum_k{\left\|w_k^m-z_m\right\|}^2<\infty,\quad m=1,\ldots,n,
\end{equation}
where $(w_k^m) ^*\Delta(\lambda_k^m)=0,\:m=1,\ldots,n;\:k\in\mathbb{Z}.$

Our further goal is to show that conditions \eqref{rownanie4.1}, \eqref{rownanie4.4} or 
\eqref{rownanie4.1}, \eqref{rownanie4.2} are 
almost sufficient for couples of sequences $\left\{\lambda\right\}, \left\{w\right\}$ 
or $\left\{\lambda\right\}, \left\{v\right\}$ to be spectral ones for the system \eqref{rownanie1.1}.

Let us consider a sequence of different complex 
numbers $\left\{\lambda_{k_0}^{m_0}\right\},\:m_0=1,\ldots,n;$ $k_0\in\mathbb{Z},$ 
satisfying \eqref{rownanie4.1}. 
We also 
assume that the index numbering of $\left\{\lambda\right\}$ is such 
that if $\lambda_{k_0}^{m_0}=\widetilde{\lambda}_k^m$ then $m_0=m,\:k_0=k.$ To begin with, however, 
we put $\lambda_{k_0}^{m_0}\neq \widetilde{\lambda}_k^m,$ for all $k,k_0\in\mathbb{Z},\:m, m_0=1,\ldots,n.$

Let now $\left\{\lambda_{k_0}^{m_0}\right\}$ be simple eigenvalues of operator $\mathcal{A}=\widetilde{\mathcal{A}}+\mathcal{B}_0\mathcal{P}_0,$ where $\mathcal{P}_0$ is given by
\eqref{rownanie1.2}, in which matrix $A_3(\theta)$ satisfies \eqref{rownanie3.6}. Then $\Delta(\lambda_{k_0}^{m_0})=0$ and let $\left\{w_{k_0}^{m_0}\right\}$ be a sequence of the left degenerating vectors of $\Delta(\lambda_{k_{0}}^{m_{0}}),$ i.e.
$$(w_{k_0}^{m_0})^*\Delta(\lambda_{k_0}^{m_0})=0.$$

We assume that the sequence $w_{k_{0}}^{m_{0}}$ satisfies \eqref{rownanie4.4}. For all indices $m_0=1,\ldots,n;$ $k_0\in\mathbb{Z},$ consider decompositions

$$(w_{k_{0}}^{m_{0}})^*=\sum_{{j}=1}^n\alpha_{jm_0}^{k_0}z_j^*.$$
Then condition \eqref{rownanie4.4} is equivalent to
\begin{equation}
\label{rownanie4.5}
\sum_{k_{0}}{\left|\alpha_{mm_{0}}^{k_0}\right|}^2<\infty,\quad m\neq m_0,\
\sum_{k_{0}}{\left|\alpha_{mm}^{k_0}-1\right|}^2<\infty, \qquad m, m_0=1,\ldots,n.
\end{equation}
Let us rewrite now relations \eqref{rownanie3.7} for $\lambda_0=\lambda_{k_{0}}^{m_{0}}$ and $w_0=w_{k_{0}}^{m_{0}}.$

We now consider the space $\ell_2$ of infinite sequences (columns) 
indexed as $\left\{a_k\right\}_{k\in\mathbb{Z}}$ with a scalar product 
defined by $\left\langle \left\{a_k\right\},\left\{b_k\right\}\right\rangle=\sum_{k}a_k\overline{b_k}.$
From the relation \eqref{rownanie3.5} we obtain that 
vectors 
$$
p_m^j=-\left\{\frac{\overline{p}_{k,m}^j}{\overline{\widetilde{\beta}}_k^m}\right\}_{k\in\mathbb{Z},\ j,\:m=1,\ldots,n}
$$
belong to $\ell_2$. 
One can also easily see that 
$$
\left\{\frac{1}{\widetilde{\lambda}_k^m-\lambda_{k_{0}}^{m_{0}}}\right\}_{k\in\mathbb{Z}}\in \ell_2,
\quad m,m_0=1,\ldots,n;\ k_0\in\mathbb{Z}.
$$
Then, putting $\lambda_0=\lambda_{k_0}^{m_{0}}$
and $w_0=w_{k_{0}}^{m_{0}}$ in the equations \eqref{rownanie3.7}, we obtain
\begin{equation}
\label{rownanie4.6}
\alpha_{m,m_{0}}^{k_{0}}=
\sum^{n}_{j=1}
\alpha_{j m_{0}}^{k_{0}}\left\langle \left\{\frac{1}{\widetilde\lambda_{k}^{m}
-\lambda_{k_{0}}^{m_{0}}}\right\}, p_{m}^{j}\right\rangle, \qquad m,m_0=1,\ldots,n;\ k_0\in\mathbb{Z}.
\end{equation}
Now we would like to rewrite relation \eqref{rownanie4.6} 
in a vector-matrix abstract form. In order to do that, we introduce a more convenient notation. 
Denote
$$ \alpha_{mm_{0}}=\left\{\alpha_{mm_{0}}^{k_0}\right\}_{{k_0}\in\mathbb{Z}},$$
$m,m_0=1,\ldots,n.$
By $S_{mm_{0}}=\left\{s_{k_{0}k}^{mm_0}\right\}_{k,\:k_0\in\mathbb{Z}},\quad m,m_0=1,\ldots,n,$ we denote infinite matrices with elements
$$s_{k_{0}k}^{mm_0}=\frac{1}{\widetilde{\lambda}_k^m-\lambda_{k_{0}}^{m_{0}}},$$ and by $A_{j m_{0}},\:j,m_0=1,\ldots,n,$ infinite diagonal matrices
$$A_{j m_{0}}=\diag\left\{\alpha_{j m_{0}}^{k_0}\right\}_{k_{0}\in\mathbb{Z}}.$$
With these notations relations \eqref{rownanie4.6} can be rewritten as
\begin{equation}
\label{rownanie4.7}
\sum_{j=1}^n A_{j m_{0}}S_{m m_{0}}p_m^j=\alpha_{mm_0},
\end{equation}
$m,m_0=1,\ldots,n.$
Now let us fix index $m$ and consider $n$ equations \eqref{rownanie4.7} with this index and $m_0=1,\!2,\ldots,n$. 
Consider another infinite diagonal matrix
$$
\Lambda_m={\diag \left\{\widetilde\lambda_k^m-\lambda_k^m\right\}}_{k\in\mathbb{Z}},
$$
and multiply  both sides of the $m$-th  equality \eqref{rownanie4.7} (for $m_0=m$) 
by this matrix from the left. 
This gives the following system of equalities 
\begin{equation}
\label{rownanie4.8}
\left\{
\begin{array}{lcl}
\displaystyle\sum_{j=1}^n A_{j m_{0}}S_{m m_{0}}p_m^j&=&\alpha_{mm_0},\qquad m_0=1,\ldots,n,\ m_0\neq m,\\
\displaystyle\sum_{j=1}^n A_{jm}\Lambda_m S_{mm}p_m^j&=&\Lambda_m\alpha_{mm},
\end{array}
\right.
\end{equation}
where we used the fact that diagonal matrices commute : $\Lambda_mA_{jm}=A_{jm}\Lambda_m$. 
Finally, we  introduce block matrix operators
$$
D_m=\left[\begin{array}{ccccc}A_{11}S_{m1}&\ldots&A_{m1}S_{m1}&\ldots&A_{n1}S_{m1}\\
\vdots&\ddots&\vdots& &\vdots\\
A_{1m}\Lambda_m S_{mm}&\ldots&A_{mm}\Lambda_m S_{mm}&\ldots&A_{nm}\Lambda_mS_{mm}\\
\vdots&&\vdots&\ddots&\vdots\\
A_{1n}S_{m n}&\ldots&A_{mn}S_{mn}&\ldots&A_{n n}S_{mn}\end{array}\right],\:m=1,\ldots,n,
$$ 
and present  \eqref{rownanie4.8} in the form
$$D_m \begin{bmatrix} p_m^1\\
\vdots\\p_m^m\\
\vdots\\p_m^n\end{bmatrix}
=\begin{bmatrix} \alpha_{m1}\\
\vdots\\\Lambda_m\alpha_{mm}\\
\vdots\\
\alpha_{mn}\end{bmatrix}.$$
Let us observe that both vectors $(p_m^1,\ldots,p_m^n)^T$ and 
$(\alpha_{m1}\ldots\Lambda_m\alpha_{mm}\ldots\alpha_{m n})^T$ 
belong to $\ell_2^n=\underbrace{_{}\ell_2\times \ell_2\times\ldots\times \ell_2}_{n \,\textrm{times}}$ 
(see \eqref{rownanie3.5},\eqref{rownanie4.5}).
Therefore the system \eqref{rownanie4.8} is solvable if and only if the 
vector $(\alpha_{m1}\ldots\Lambda_m\alpha_{mm}\ldots\alpha_{m n})^T$ belongs 
to the image of operator $D_m$ as an operator from $\ell_2^n$ to $\ell_2^n.$ In the sequel
 we show that, for all sequences $\left\{\lambda\right\}$ satisfying, \eqref{rownanie4.1} 
and for almost all sequences $\left\{w\right\}$ satisfying \eqref{rownanie4.4}, 
operators $D_m, m=1,\ldots,n,$ are bounded and with bounded inverse operators from $\ell_2^n$ to $\ell_2^n$. 
This means that the spectral assignment problem is solvable.
In the further argument we use the following
\begin{proposition}\label{statement2}
Let  $\left\{\lambda_k\right\}_{k\in\mathbb{Z}}$
be a sequence such that 
 $$
 \sum_{k \in \Zset}{\left|\lambda_k-a+\ii(b+2\pi k)\right|}^2<\infty,
 $$
for some $a, b\in \mathbb{R}$. Then the family $\left\{\e^{\lambda_k t}\right\}_{k\in\mathbb{Z}}$ forms a 
Riesz basis in $L_2(0,1)$.
\end{proposition}
There are several ways to prove this classic result (see \cite{Avdonin_Ivanov_1995}). 
It may be  be obtained, for example,  from the Paley-Wiener theorem \cite{Paley-Wiener} and Lemma II.4.11 
\cite{Avdonin_Ivanov_1995}.

Next we prove the following preliminary result.
\begin{lemma}\label{Lemma1}
1. For $m\neq m_o$ operators $S_{mm_{0}}$ are bounded  as operators from 
$\mathcal{L}(\ell_2)$  and have bounded inverses.
2. $\Lambda_mS_{mm}$ is a bounded operator from $\mathcal{L}(\ell_2)$ and has a bounded inverse.
\end{lemma}
\begin{proof} Let $\left\{\varphi_k\right\},\left\{\tilde{\varphi}_k\right\},k\in\mathbb{Z}$, be 
two Riesz basis of a Hilbert space $H$ and let $R$ be a bounded operator with a bounded inverse,
 such that $R\varphi_k=\tilde{\varphi}_k,\:k\in\mathbb{Z}.$

For $f\in H$ we have
$$f=\sum_{j}a_j\varphi_j,\qquad Rf=\sum_{j}a_jR\varphi_j$$

Then
$$
R\varphi_j
=\widetilde{\varphi}_j=\sum_{k}c_{jk}\varphi_k
=\sum_{k}\left\langle \widetilde{\varphi}_j,\psi_k\right\rangle\varphi_k,\quad j\in\mathbb{Z},
$$
where ${\left\{\psi_k\right\}}_{k\in\mathbb{Z}}$ is the bi-orthogonal with respect basis to ${\left\{\varphi_k\right\}}_{k\in\mathbb{Z}}.$ 
Hence
$$
Rf=\sum_{j}a_j\sum_{k}\left\langle \widetilde{\varphi}_j,\psi_k\right\rangle\varphi_k
=\sum_{k}\sum_{j}a_j\left\langle \widetilde{\varphi}_j,\psi_k\right\rangle\varphi_k=\sum_{k}b_k\varphi_k,$$
where $b_k=\sum_{j}a_j\left\langle
\widetilde{\varphi}_j,\psi_k\right\rangle.$

This means  that the infinite matrix $\widehat{R}$ corresponding to $R$ in
 the basis $\left\{\varphi_k\right\}$ is of the form
$$
\widehat{R}=\left\{\widehat{r}_{kj}=
\left\langle \widetilde{\varphi}_j,\psi_k\right\rangle\right\}_{\stackrel{k\in\mathbb{Z}}{j\in\mathbb{Z}}},
$$
where
$$
\widehat{R}\begin{bmatrix} \vdots \cr
a_{-1}\cr
 a_0\cr 
a_1\cr
 \vdots\cr
\end{bmatrix}
=
\begin{bmatrix} \vdots \\ 
b_{-1}\cr
 b_0\cr
b_1\cr
\vdots\cr
\end{bmatrix},\;\{a_j\},\:\{b_j\} \in \ell_2.
$$
Let now $H=L_2(0,1)$ and $\left\{\widetilde{\varphi}_k\right\}_{k\in\mathbb{Z}}$ 
be a Riesz basis of the form $\widetilde{\varphi}_k=\e^{{\widetilde{\lambda}_k^m}t}$ 
for some $m=1,\ldots,n$. 
Let now $\left\{\varphi_k\right\}$ be a Riesz basis which is bi-orthogonal to 
$\left\{\psi_k=\e^{-\overline{\lambda_k^{m_{0}}}t}\right\}_{k\in\mathbb{Z}}$
for some $m_{0}=1,\ldots,n$ 
(the fact that $\left\{\psi_k\right\}_{k\in\mathbb{Z}}$ is a Riesz basis of
 $L_2(0,1)$ 
follows from Proposition~\ref{statement2}).
One has
\begin{eqnarray*}
\left\langle \widetilde{\varphi}_j,\psi_k\right\rangle&=&\int^{0}_{-1}\e^{\widetilde{\lambda}_j^m t}\e^{-{\lambda_k^{m_0}}t}dt=\frac{1}{{\widetilde{\lambda}_j^m}-{\lambda}_k^{m_0}}\left( \e^{\widetilde{\lambda}_j^m-{\lambda_k^{m_0}}}-1\right)\\
&=&\frac{1}{\widetilde{\lambda}_j^m-{\lambda}_k^{m_0}}\left(\mu_m \e^{-{\lambda_k^{m_0}}}-1\right),
\end{eqnarray*}
i.e. $\hat{R}=\left\{\hat{r}_{kj}\right\}_{\stackrel{k\in\mathbb{Z}}{j\in\mathbb{Z}}}$,
 where
$$
\widehat{r}_{k j}=s_{k j}\left(\mu_m {\e^{-\lambda_k^{m_{0}}}}-1\right),\:k,\;j\in\mathbb{Z}.
$$
Thus
$$
\widehat{R}=\varepsilon_{m m_{0}}S_{mm_{0}},
$$
where $\varepsilon_{m m_{0}}$ is the infinite matrix
$$
\varepsilon_{m m_{0}}=
\left[\begin{array}{ccccc}
\ddots & \ldots & \ldots & \ldots & \ldots\\
\vdots & \mu_m \e^{-{\lambda_{-1}^{m_0}}}-1 & 0 & 0 & \vdots \\
\vdots & 0 & \mu_m \e^{-{\lambda_0^{m_0}}}-1 & 0 & \vdots \\
\vdots & 0 & 0 & \mu_m \e^{-{\lambda_1^{m_0}}}-1 & \vdots \\
\ldots & \ldots & \ldots & \ldots & \ddots\\
%
\end{array}\right].
$$
Hence, we have the following alternative:\\
1. If $m_0\neq m$, then the sequence 
$\left\{\mu_m\e^{-\lambda_k^{m_{0}}}-1\right\}_{k\in\mathbb{Z}}$ is bounded 
and separated from $0,$ i.e. $\varepsilon_{mm_{0}} : \ell_2\rightarrow \ell_2$ is a  bounded opeartor 
 with a bounded 
inverse. Hence,
\begin{equation}
\label{rownanie*}
S_{mm_{0}}=\varepsilon_{mm_{0}}^{-1}\hat{R}.
\end{equation}
2. If $m=m_{0}$, then $\mu_m \e^{-\lambda_k^m}\rightarrow1,\:k\rightarrow\infty,$
moreover
\begin{eqnarray*}
\mu_m \e^{-\lambda_k^m}-1 &=&\e^{\widetilde{\lambda}_k^m-\lambda_k^m}-1\\
&=&\left(1+\left(\widetilde{\lambda}_k^m-\lambda_k^m\right)+\ldots+
\frac{\left(\widetilde{\lambda}_k^m-\lambda_k^m\right)^s}{s!}+\ldots\right)-1\\
&=&\left(\widetilde{\lambda}_k^m-\lambda_k^m\right)\left(1+\bar{o}
\left(\widetilde{\lambda}_k^m-\lambda_k^m\right)\right).
\end{eqnarray*}
Therefore,
$$
\varepsilon_{mm}=\Lambda_m Q_m=Q_m\Lambda_m,
$$
where $Q_m=\diag \left(1+\bar{o}(\widetilde{\lambda}_k^m-\lambda_k^m)\right)_{k\in\mathbb{Z}}$ 
has a bounded inverse, so
\begin{equation}
\label{rownanie**}
\Lambda_mS_{mm}=Q_m^{-1}\hat{R}
\end{equation}
From \eqref{rownanie*}, \eqref{rownanie**} it follows that $S_{m m_{0}}, \ m\neq m_0$ 
and  $\Lambda_mS_{mm}$ are bounded and have bounded inverse.
\end{proof}

\begin{remark}
In our previous consideration we assumed implicitly that our sequences 
$\left\{\lambda_{k_{0}}^{m_{0}} \right\}_{k_{0}\in\mathbb{Z}}$ are different from
$\left\{\widetilde\lambda_k^m \right\}_{k\in\mathbb{Z}},$ i.e.
$\widetilde{\lambda}_k^m\neq{\lambda}_{k_{0}}^{m_{0}}$ 
for all $k,k_0\in\mathbb{Z},\ m,m_{0}\in\left\{1,\ldots,\:,n\right\}$ in particular
$\lambda_k^m\neq\widetilde{\lambda}_k^m$ for every $k\in\mathbb{Z},\:m\in\left\{1,\ldots,n\right\}.$
\end{remark}
\bigskip

Now let us allow $\lambda_k^m=\widetilde{\lambda}_k^m$ for some indices $k\in I\subset\mathbb{Z}.$
Note that in this case the operators $S_{m m_{0}},\;m_{0}\neq m$, 
are still well-defined and the operator $\Lambda_m S_{mm}$ can be well-defined as well if we 
define its components as limits of correspondent components when 
$\lambda_k^m\rightarrow\widetilde{\lambda}_k^m,\ k\in I$.
 This means that for $k\in I$ all non-diagonal 
elements of the $k$-th line of $\Lambda_m S_{mm}$ equal $0$ and the diagonal 
elements equal $1.$ Besides, $S_{m m_{0}},\;m_0\neq m$ and $\Lambda_m S_{mm}$ remains bounded and with a 
bounded inverse $\ell_2\rightarrow \ell_2$ operators since formulas  \eqref{rownanie*} \eqref{rownanie**} 
remain true also when $\lambda_k^m=\widetilde{\lambda}_k^m.$ 
Finally if we consider the 
dependence $\Lambda_m S_{mm}$  of the sequence $\left\{\lambda_k^m\right\}_{k\in\mathbb{Z}}$, 
 one can easy prove that
$$
\Lambda_m S_{mm}\left(\left\{{{\lambda'}_k^m}\right\}\right)
\rightarrow\Lambda_m S_{mm}\left(\left\{\lambda_k^m\right\}\right)
$$ 
as 
$$
\sum_{k}{\left|{\lambda'}_k^m-\lambda_k^m\right|}^2\rightarrow 0.
$$
In other words this means that operators $\Lambda_m S_{mm}$ and, as a consequence also its  
inverse operators ${\left(\Lambda_m S_{mm}\right)}^{-1},$ depend continuously of 
sequence ${\left\{\lambda_k^m\right\}}_{k\in\mathbb{Z}}$ on the set 
$$
\left\{\left\{\lambda_k^m\right\}:\sum_{k}{\left|\lambda_k^m-\widetilde{\lambda}_k^m\right|}^2<\infty\right\}.
$$
Now we are ready to prove our main results on the spectral assignment.
\bigskip
\begin{theorem}\label{Theorem1}

Let $\mu_1$, $\mu_2$, $\ldots$, $\mu_n$ be different nonzero complex numbers and $z_1$, $z_2$, $\ldots$, $z_n$
be nonzero n-dimensional linear independent vectors. Denote

$$
\widetilde{\lambda}_k^m=\ln\left|\mu_m\right|+\ii\left(\Arg \mu_m+2\pi k\right),\quad m=1,\ldots,n,\:k\in
\mathbb{Z}.
$$
Let us consider an arbitrary sequence of different complex numbers 
$\left\{\lambda_k^m\right\}_{\stackrel{k\in\mathbb{Z}}{m=1,\ldots,n}}$ such that
$$
\sum_{k}{\left|\lambda_k^m-\widetilde{\lambda}_k^m\right|}^2<\infty,\quad m=1,\ldots,n.
$$
Then there is a small enough $\varepsilon>0$ such for any sequence of nonzero vectors 
$\left\{d_k^m\right\}_{\stackrel{k\in\mathbb{Z}}{m=1,\ldots,n}}$ satisfying
$$
\sum_{k}{\left\|d_k^m-z_m\right\|}^2<\varepsilon, \quad m=1,\ldots,n
$$
one can choose matrices $A_{-1},A_{2}(\theta),A_3(\theta)$ such that for the system \eqref{rownanie1.1}, 
with these matrices, the following two conditions hold: 
\begin{enumerate}
 \item[i)] all the numbers $\left\{\lambda_k^m\right\}$ are 
roots of the characteristic equation $\det\Delta(\lambda_k^m)=0$, 
$k\in\mathbb{Z}, m=1,\dots, n$,
\item[ii)] ${d_k^m}$ are right degenerating vectors for $\Delta(\lambda_k^m)$ :
${d_k^m}^*\Delta(\lambda_k^m)=0,\;m=1,\ldots,n;\ k\in\mathbb{Z}$.
\end{enumerate}
Such a choice is unique if we put the following additional condition on matrix $A_3(\theta)$:
$$
(C)   \qquad \qquad \qquad             \int^{0}_{-1}A_3(\theta)\dd \theta=0.
$$

\end{theorem}

\begin{proof}First we denote by $A_{-1}$ the matrix uniquely defined by the relations:
$$
z_m^*A_{-1}
=\mu_m z_m^*,\:m=1,\ldots,n,
$$ 
and denote by $\left\{y_j\right\}_{j=1,\ldots,n}$ 
the bi-orthogonal basis with respect  to $\left\{z_j\right\}_{j=1,\ldots,n}$ in $\mathbb{C}^n$.
Then the corresponding operator $\widetilde{\mathcal{A}}$, generated by the matrix $A_{-1}$,
has eigenvalues $\widetilde{\lambda}_k^m,\ m=1,\ldots,n,\ k\in\mathbb{Z};\:\widetilde{\lambda_0}=0$ 
and possesses the Riesz basis of eigenvectors of 
$\left\{\widetilde{\varphi}_k^m\right\}_{k\in\mathbb{Z}}\cup\left\{\widetilde{\varphi}_j^0\right\}_{j=1,\ldots,n}.$ 
In the case when all $\mu_m\neq1,\:m=1,\ldots,n$, 
the corresponding  eigenvectors of $\widetilde{\mathcal{A}}$ 
are 
\begin{equation}
\label{rownanie(a)}
\left\{
\begin{array}{l}
\displaystyle{\widetilde{\varphi}}_k^m=
\begin{pmatrix}
0\cr
\e^{\widetilde{\lambda}_k^m\theta}y_m
\end{pmatrix},\quad m=1,\ldots,n;\ k\in\mathbb{Z},\quad 
\widetilde{\mathcal{A}}\widetilde{\varphi}_k^m=\widetilde{\lambda}_k^m\widetilde{\varphi}_k^m,\\[2ex]
\displaystyle\widetilde{\varphi}_j^0=
\begin{pmatrix}
y_j-A_{-1}y_j\cr
y_j
\end{pmatrix},
\quad j=1,\ldots,n;\quad\widetilde{\mathcal{A}}{\widetilde{\varphi}_j^0}=0.
\end{array}
\right.
\end{equation}
In the case  $\mu_1=1$, the vector
$$
\varphi_0^1=
\begin{pmatrix}
  A_{-1}y_j \cr
\theta y_j           
\end{pmatrix}, \quad j=1,\ldots,n;
$$
is a  generalized eigenvector  of $\widetilde{\mathcal{A}}$ corresponding to $\widetilde{\lambda}_0$ and 
the other $\widetilde{\varphi}_k^m,\;\widetilde{\varphi}_j^0$ are given by formula \eqref{rownanie(a)}. 
Let us show that there is a choice of a bounded operator 
$\mathcal{P}_0:X_{\widetilde{\mathcal{A}}}\rightarrow\mathbb{C}^n$ 
(or equivalently a choice of matrices $A_2(\theta),A_3(\theta))$ such that
\begin{equation}
\label{rownanie(a1)}
\lambda_{k_{0}}^{m_{0}}\in\sigma(\widetilde{\mathcal{A}}+\mathcal{B}_0\mathcal{P}_0),
\end{equation}
or equivalently $\det\Delta(\lambda_{k_{0}}^{m_{0}})=0,$ and
\begin{equation}
\label{rownanie(a2)}
d_{k_{0}}^{{m_{0}}*}\Delta(\lambda_{k_{0}}^{m_{0}})=0,\quad m=1,\ldots,n;\ k\in\mathbb{Z}.
\end{equation}
We represent vectors $d_{k_{0}}^{m_{0}}$ in the basis $\left\{{z_j}\right\}_{j=1,\ldots,n},$ namely
$$
d_{k_{0}}^{m_{0}}=\sum_{m=1}^n \alpha_{m m_{0}}^{k_{0}}z_m,\quad k_{0}\in\mathbb{Z};\ m,m_{0}=1,\ldots,n.
$$
With these notations the condition

$$
\sum_{k_{0}}{\|d_{k_{0}}^m-z_m\|}^2<\infty,\quad m=1,\ldots,n
$$
implies
$$
\sum_{k_{0}}|\alpha_{mm_{0}}^{k_{0}}|^2<\infty,\quad m\neq m_{0};
\qquad \sum_{k_{0}}{|\alpha_{mm}^{k_{0}}-1|}^2<\infty,\;m=1,\ldots,n.
$$
and these sums tend to zero as
\begin{equation}
\label{rownanie(b)}
\sum_{k_{0}}{\|d_{k_{0}}^m-z_m\|}^2\rightarrow 0.
\end{equation}
Therefore, under condition \eqref{rownanie(b)} operators $D_m,\;m=1,\ldots,n$ 
tend to block diagonal operators
$$
\widehat{D}_m=\begin{bmatrix}
                       S_{m1} & 0 & \ldots & 0 & \ldots & 0 \cr
                       \vdots & \vdots & \vdots & \vdots &  & \vdots \\
                       0 & 0 & \ldots & \Lambda_mS_{mm} & \ldots & 0 \\
                       \vdots & \vdots & \vdots & \vdots & \vdots & \vdots \\
                       0 & 0 &  \ldots & 0 & \ldots & S_{mn}
                     \end{bmatrix}
$$
which have a bounded inverse due to Lemma~\ref{Lemma1}. This means that for a
small enough $\varepsilon >0$, the  inequality $\sum\limits_{k}{\|d_k^m-z_m\|}^2<\varepsilon$ 
implies that operators $D_m,\;m=1,\ldots,n$ have bounded inverse. If numbers $\{{\lambda}_k^m\}$  
are different from $\{{{\widetilde{\lambda}}_k^m}\}$ the later fact yields the existence 
of a bounded operator {$\mathcal{P}_0:X_{\mathcal{A}}\rightarrow{\mathbb{C}^{n}}$} for which 
the relations \eqref{rownanie(a1)},\eqref{rownanie(a2)} are satisfied. 
Besides, this operator is unique if we require 
additionally: $\mathcal{P}_0\widetilde{\varphi}_j^0=0,\;j=1,\ldots,n$,
 which  is equivalent to condition (C). If we allow 
coincidence $\lambda_k^m=\widetilde{\lambda}_k^m$ for some indices $\{k,m\}\in\mathbb{I}$
 one needs to use continuous dependence of operators $\mathcal{D}$ 
on the sequence $\{\lambda_k^m\}$ (see Remark to Lemma~\ref{Lemma1}).
We approximate 
$\{{\lambda'}_k^m\}$ by $\{{\lambda'}_k^m\}\;(\sum\limits_k{|{\lambda'}_k^m-\lambda_k^m|}^2\rightarrow 0)$ 
such that $\{{\lambda'}_k^m\}\neq\widetilde{\lambda}_k^m$.
Since the conditions \eqref{rownanie(a1)}, \eqref{rownanie(a2)}  
are satisfied for operator $\mathcal{P}_0(\{{\lambda'}_k^m\})$ 
they are also satisfied for $\mathcal{P}_0(\{{\lambda}_k^m\}).$
This completes the proof
\end{proof}
\begin{lemma}\label{Lemma2}
Let $\{\lambda_k^m\}$ be a given sequence such 
that $\sum\limits_k{|\lambda_k^m-\widetilde{{\lambda}}_k^m|}^2<\infty$ 
and $\{\widehat{\alpha}_{mm_{0}}^{k_{0}}\}$ 
be an arbitrary sequence satisfying
$$
\sum\limits_{k_{0}}{|\widehat{\alpha}_{mm_{0}}^{k_{0}}|}^2<\infty,\quad m\neq m_{0};
\qquad \sum\limits_{k_{0}}{|\widehat{\alpha}_{mm}^{k_{0}}-1|}^2<\infty.
$$
Then for any $\varepsilon>0$ and $m=1,\ldots,n$ there is a sequence 
$\{{\alpha}_{mm_{0}}^{k_{0}}\}$ satisfying
$$
\sum\limits_{k_{0}}{|{\alpha}_{mm_{0}}^{k_{0}}-\widehat{\alpha}_{mm_{0}}^{k_{0}}|}^2<\varepsilon
$$
and such that the operator $D_m$ has a bounded inverse.
\end{lemma}
\begin{proof}
First, for given $\{\lambda_k^m\}$, let us choose $\varepsilon_{0}>0$ such  that $D_m$ will
 be invertible for 
$$
\sum\limits_{k_{0}}{|{\alpha}_{mm_{0}}^{k_{0}}|}^2<\varepsilon_{0},
\quad m\neq m_{0},\qquad \sum\limits_{k_{0}}{|\alpha_{mm}^{k_{0}}-1|}^2<\varepsilon_{0}.
$$
Then, one can find  a great enough $N$ such that
$$
\sum_{|k_{0}|>N}{|\widehat{\alpha}_{mm_{0}}^{k_{0}}|}^2<\varepsilon_{0},\quad
\sum_{|k_{0}|>N}{|\widehat{\alpha}_{mm}^{k_{0}}-1|}^2<\varepsilon_{0}.
$$
Next we consider the sequences 
$\{{\alpha}_{mm_{0}}^{k_{0}}\}$ for which 
\begin{equation}
\label{rownanie(c)}
{\alpha}_{mm_{0}}^{k_{0}}=\widehat{\alpha}_{mm_{0}}^{k_{0}},\;as\; |k_{0}|>N. 
 \end{equation}

Our goal is to choose the remaining components ${\alpha}_{mm_{0}}^{k_{0}},\;|k_{0}|\leq N$ 
in order to satisfy the requirement of Lemma.
Denote rows  of matrix $D_m$ by $({\ell}_{k_{0}}^{m_{0}})^*,\;\ m_{0}=1,\ldots,n;\ k_{0}\in\mathbb{Z}$ 
and let $q_{k_{0}}^{m_{0}}$ be the correspondent components of the  vector 
$q=D_m p$,  i.e. 
$q_{k_{0}}^{m_{0}}=({\ell}_{k_{0}}^{m_{0}})^*p,\;p\in{\ell}_2^n\; m_{0}=1,\ldots,n;\;k_{0}\in\mathbb{Z}.$
The space $L=\underbrace{\ell_2\times\ldots\times \ell_2}_{n \textrm{ times}}$ 
may be written as $L=L^{1}\oplus L^{2},$
where 
$$
L^{1}=\{q:{q_{k_{0}}^{m_{0}}=0},\ m_0=1,\ldots,n,\ |k_{0}|>N\},
$$
and
$$ 
L^2=\{q:{q_{k_{0}}^{m_{0}}=0},\ |k_{0}|\leq N,\ m_0=1,\ldots,n\}. 
$$
Let $P$ be the orthogonal projector on $L_2$.
Let us observe that the lines $({\ell}_{k_{0}}^{m_{0}})^*$ for $|k_{0}|>N$ do not depend on chosen 
components $\alpha_{mm_{0}}^{k_{0}},\ \left|k_{0}\right|\leq N$ and that if we put  
${\alpha}_{mm_{0}}^{k_{0}}=\delta_{mm_{0}}, |k_{0}|\leq N,\ m,\:m_0=1,\ldots,n$, then 
the operator ${D}_m: L\rightarrow L$ has a bounded inverse. This means that 
for all sequences $\{\alpha_{mm_0}^{k_{0}}\}$ satisfying \eqref{rownanie(c)} we have 
 $PD_{m}L=L^2$ and the
 invertibility of $D_m$ occurs if and only if
\begin{equation}
\label{rownanie(!)}
D_{m}L\supset L_1.
\end{equation}
Let $L^{1'}$ be the subspace 
$$
\left\{p\in L:({\ell}_{k_0}^{m_0})^*p=0 ,\ |k_{0}|>N,\;m_0=1,\ldots,n\right\}
$$
of dimension $(2N+1)n$.
 Denote by ${\ell_k^j}', j=1,\ldots,n;\ \left|k\right|\leq N$ 
a basis of $L^{1'}$. With this notations one can see that
\eqref{rownanie(!)} is equivalent to the invertibility of $(2N+1)n\times(2N+1)n$ 
matrix
$$
M=\left\{({\ell}_{k_{0}}^{m_{0}})^*{\ell_k^j}', \ |k_{0}|\leq N,\;|k|\leq N,
\ m_{0}=1,\ldots,n,\ j=1,\ldots,n\right\}
$$
i.e. $\det M\neq 0$.
The components of $M$ are linear functions of chosen $\alpha_{mm_{0}}^{k_{0}},\ |k_{0}|\leq N$. 
Therefore, its determinant is a polynomial of these coefficients. Besides, $\det M$ is not identical zero
because matrix $D_m$ is invertible if we take  $\alpha_{mm_{0}}^{k_{0}}=\delta_{mm_{0}}$,  i.e.
 $\det M\neq 0.$ This implies that $M$ is invertible almost everywhere in $\mathbb{C}^{2(N+1)n}$. 
This fact completes the proof of Lemma because we can choose $\alpha_{mm_{0}}^{k_{0}},\ |k_{0}|\leq N$,  
in such a way that 
$$
\sum\limits_{|k_{0}|\leq N}\mid\alpha_{mm_{0}}^{k_{0}}-\widehat{\alpha}_{mm_{0}}^{k_{0}}\mid<\varepsilon
$$ 
and the operator $D_m$ will be invertible.
\end{proof}
\begin{remark}
From the proof of Lemma~\ref{Lemma2} it is easy to see that actually sequence $\alpha_{mm_{0}}^{k_{0}}$ 
may differ from $\widehat{\alpha}_{mm_{0}}^{k_{0}}$ only for a finite number of components 
$|k_{0}|\leq N.$
\end{remark}
Due to Lemma~\ref{Lemma2} the formulation of Theorem~\ref{Theorem1} may be generalized 
in the following way.
\begin{theorem}\label{Theorem2}
Let the sequence $\{\lambda_k^m\}_{k\in\mathbb{Z},\;m=1,\ldots,n}$ and vectors $z_m,\;m=1,\ldots,n$
 be chosen according the assumptions of Theorem~\ref{Theorem1}.  Then for any sequence of 
vectors\\ $\{\widehat{d}_k^m\}_{k\in\mathbb{Z},\;m=1,\ldots,n}$ satisfying
$$
\sum\limits_k{\|\widehat{d}_k^m-z_m\|}^2<\infty,\; m=1,\ldots, n,
$$
and  for any $\varepsilon>0,$ there is a sequence $\{d_k^m\}_{k\in\mathbb{Z},\;m=1,\ldots,n}:$
$$\sum\limits_k{\|d_k^m-\widehat{d}_k^m\|}^2<\varepsilon$$
such that, for some choice of matrices $A_{-1}, \ A_2(\theta),\  A_3(\theta)$, satisfying
 $\int^{0}_{-1}A_3(\theta)\dd\theta=0$, 
the conditions i), ii) of Theorem~\ref{Theorem1} 
are verified.  Moreover, $\{d_k^m\}$ may be chosen in such 
a way, that $d_k^m=\widehat{d}_k^m$ for all $|k|>N$ and for some $N\in\mathbb{N}.$
\end{theorem}
And then we obtain the following result.
\begin{theorem}\label{Theorem3}
Let the sequences 
$\{\lambda_k^m\}_{k\in\mathbb{Z},\;m=1,\ldots,n}$ and $\{\widehat{d}_k^m\}_{k\in\mathbb{Z},\;m=1,\ldots,n}$
 be from Theorem~\ref{Theorem2}.
Let, in addition, the complex numbers $\lambda_j^0,\;j=1,\ldots,n$ be different from each other 
and different from $\lambda_k^m$ and let $d_j^0,\;j=1,\ldots,n$ be linear independent vectors.
Then, for any $\varepsilon>0$ there exist $N>0,$ a sequence ${\{d_k^m\}}_{k\in\mathbb{Z},\;m=1,\ldots,n}:$
$$\sum\limits_k{\|d_k^m-\widehat{d}_k^m\|}^2<\varepsilon,\;d_k^m=\widehat{d}_k^m,\;\mid k\mid>N,\;m=1,\ldots,n$$
and a choice of matrices $A_{-1}, A_2(\theta), A_3(\theta)$ such that:
\begin{enumerate}
\item[i)] all the numbers $\{\lambda_k^m\}_{k\in\mathbb{Z},\;m=1,\ldots,n}\cup\{\lambda_j^0\}_{j=1.\ldots,n}$ 
are roots of the characteristic equation $\det \Delta(\lambda)=0$;
\item[ii)] $d_k^{m^*}\Delta(\lambda_k^m)=0,\;m=1,\ldots,n,\;k\in\mathbb{Z}$ and $d_j^{0^*}\Delta(\lambda_j^0)=0.$
\end{enumerate}
\end{theorem}
\begin{proof}
Denote by $C$ a $(n\times n)$ matrix uniquely defined by the equalities:
$$ 
d_j^{0^*}C=\lambda_j^0d_j^{0^*},\;j=1,\ldots,n,
$$
and let us put 
$$
\widehat{f}_k^{m^*}=\widehat{d}_k^{m^*}(I-\frac{1}{\lambda_k^m}C)^{-1},\;k\in\mathbb{Z} ,\;m=1,\ldots,n
$$
if
$\lambda_k^m\neq 0$ 
and 
$$
\widehat{f}_k^{m^*}=\widehat{d}_k^{m^*}C^{-1}
$$ 
for $\lambda_k^m=0.$

It is easy to see that the sequences $\{\widehat{f}_k^m\}$ are also quadratically closed to $z_m$:
$$
\sum_k\|\widehat{f}_k^m-z_m\|^2<\infty.
$$
Therefore,  due to Theorem~\ref{Theorem2}, 
for any $\varepsilon>0$, there exist matrices 
$A_{-1},\widehat{A}_2(\theta),\widehat{A}_3(\theta) \\(\int^{0}_{-1}\widehat{A}_3(\theta)\dd\theta=0)$, a 
number $N>0$ and a sequence of vectors $\{f_k^m\}_{k\in\mathbb{Z},\;m=1,\ldots,n}$:
$$
\sum\limits_k\|f_k^m-\widehat{f}_k^m\|^2<\frac{\varepsilon}{M^2},\qquad f_k^m=\widehat{f}_k^m,\;\mid k\mid>N
$$  
such that
$$
f_k^{m*}\widehat{\Delta}(\lambda_k^m)=0,\; k\in\mathbb{Z},\;m=1,\ldots,n;
$$
where $M=\sup\{\|I-\frac{1}{\lambda_k^m}C\|,\;\lambda_k^m\neq0;\;\|C\|\}$ and 
$$
\widehat{\Delta}(\lambda)=
\lambda I-\lambda \e^{-\lambda}A_{-1}+\int^{0}_{-1}\lambda \e^{\lambda\theta}\widehat{A}_2(\theta)\dd\theta+
\int^{0}_{-1}\e^{\lambda\theta}\widehat{A}_3(\theta)\dd\theta.
$$
Now, let us put
$$
\Delta(\lambda)=(I-\frac{1}{\lambda}C)\widehat{\Delta}(\lambda),\;\lambda\neq 0.
$$
One can note that $\widehat{\Delta}(0)=0$ and the function 
$\widehat{\Delta}_1(\lambda)=\frac{1}{\lambda}\widehat{\Delta}(\lambda)$
 may be extended to zero by the formula
$$
\widehat{\Delta}_1(0)=I-A_{-1}+\int^{0}_{-1}\widehat{A}_2(\theta)\dd\theta+
\int^{0}_{-1}\lim_{\lambda\rightarrow 0}{\frac{\e^{\lambda\theta}-1}{\lambda}}\widehat{A}_3(\theta)\dd\theta 
$$
Then,  one can define
$$\Delta(0)=-C\widehat{\Delta}_1(0).$$
Let us observe that $\Delta(\lambda)$ can be written as
$$
\Delta(\lambda)=\lambda I-\lambda \e^{-\lambda}A_{-1}+
\int^{0}_{-1}\lambda \e^{\lambda\theta}\widehat{A}_2(\theta)\dd\theta+
\int^{0}_{-1}\e^{\lambda\theta}\widehat{A}_3(\theta)\dd\theta-C+\e^{-\lambda}CA_{-1}
$$
$$
\hskip -15ex
-\int^{0}_{-1}\e^{\lambda\theta}C\widehat{A}_2(\theta)\dd\theta+\int^{0}_{-1}\e^{\lambda\theta}
\int^{\theta}_{-1}C\widehat{A}_3(\tau)\dd\tau\dd\theta,
$$ 
$$
\hskip -12ex =\lambda I-\lambda \e^{-\lambda}A_{-1}+
\int^{0}_{-1}\lambda \e^{\lambda\theta}{A}_2(\theta)\dd\theta+
\int^{0}_{-1}\e^{\lambda\theta}{A}_3(\theta)\dd\theta
$$ 
for $A_2(\theta)$ and $A_3(\theta)$ given by
\begin{eqnarray*}
A_2(\theta)&=&\widehat{A}_2(\theta)-(\theta+1)C-\theta CA_{-1},\cr
A_3(\theta)&=&\widehat{A}_3(\theta)+\int^{\theta}_{-1}C\widehat{A}_3(\tau)\dd\tau-C-CA_{-1}
\end{eqnarray*}
It remains to note that, with this choice of matrices, the conditions i), ii) are satisfied. 
Indeed:
$$
\det\Delta(\lambda)=\det(I-\frac{1}{\lambda}C)\det\widehat{\Delta}(\lambda),
$$
so all the numbers $\{\lambda_k^m\}$ and $\{\lambda\}_j^0$ are roots of the characteristic equation 
and
$$
\begin{array}{rcl}
d_k^{m*}\Delta(\lambda_k^m)&=&f_k^{m*}\widehat{\Delta}(\lambda_k^m),\qquad k\in\mathbb{Z},\ m=1,\ldots,n,\\
&=&0,\\
d_j^{0*}\Delta(\lambda_j^0)&=&0,\qquad j=1,\ldots, m
\end{array}
$$
and finally, for all $m=1,\dots,n$, we have
$$
\sum\limits_k{\|d_k^m-\widehat{d}_k^m\|}^2\leq\sum\limits_{k,m}M^2\|f_k^m-\widehat{f}_k^m\|<\varepsilon.
$$
This completes the proof.
\end{proof}
\section{Conclusion}
We give here some conditions on sets of complex numbers $\{\lambda\}$ and $n$-vectors $\{d\}$ 
such that they form a spectral set for a neutral type systems. This is a first etap for solving
vector moment problems using the exact controllability properties of a neutral type system related
to the given moment problem.
 
\end{document}